\documentclass{amsart}
\title{Update: Remarks on Countable Tightness} 
\author{Marion Scheepers}

\usepackage{amsfonts}
\newtheorem{theorem}{{\bf Theorem}}

\newtheorem{lemma}[theorem]{{\bf Lemma}}

\newcommand{\naturals}{{\mathbb N}}

\newcommand{\sone}{{\sf S}_1}
\newcommand{\gone}{{\sf G}_1}

\date{September 28, 2015} 

\usepackage{color}

\subjclass[2000]{ 54A35, 54D65}
\keywords{Selection principle, countable strong fan tightness, infinite game}
\address{Department of Mathematics\\ Boise State University\\ Boise, Idaho 83725}
\email{mscheepe@boisestate.edu}

\begin{document}
\begin{abstract}
The proof of Theorem 11 of the paper \cite{RCT} relies on Lemma 10 of that paper. The offered proof of Lemma 10 had shortcomings, and I was recently asked for details. This note gives an alternative, complete proof of \cite{RCT}, Lemma 10. 
\end{abstract}
\maketitle

In \cite{RCT} we considered the selection principle $\sone(\mathcal{A},\mathcal{B})$ for specific instances of families $\mathcal{A}$ and $\mathcal{B}$ of sets. Recall that $\sone(\mathcal{A},\mathcal{B})$ denotes the statement that there is for each sequence $(O_n:n\in\naturals)$ of elements of $\mathcal{A}$ a corresponding sequence $(x_n:n\in\naturals)$ such that for each $n$ we have $x_n\in O_n$, and $\{x_n:n\in\naturals\}\in\mathcal{B}$.

There is a natural game, denoted $\gone(\mathcal{A},\mathcal{B})$ associated with this selection principle: This two-player game is played as follows. There is an inning for each $n<\omega$. In inning $n$ player ONE selects and $O_n\in\mathcal{A}$, and then TWO responds by selecting an $x_n\in O_n$. A play $(O_0,\; x_0,\; \cdots,\; O_n,\; x_n\cdots)$ is won by player TWO if the set $\{x_n:n<\omega\}$ is an element of the family $\mathcal{B}$; else, the play is won by player ONE.

If player ONE does not have a winning strategy in the game $\gone(\mathcal{A},\mathcal{B})$, then the selection principle $\sone(\mathcal{A},\mathcal{B})$ holds of the pair $\mathcal{A},\; \mathcal{B}$. For many topological families $\mathcal{A}$ and $\mathcal{B}$ it is the case that under appropriate circumstances also the converse holds: 
The selection principle $\sone(\mathcal{A},\mathcal{B})$ implies that ONE has no winning strategy in the game $\gone(\mathcal{A},\mathcal{B})$. Theorem 11 of \cite{RCT} was intended to demonstrate an extreme case of failure of this converse for well-studied examples of $\mathcal{A}$ and $\mathcal{B}$. 

More precisely: In \cite{Sakai} Sakai defined the notion of countable strong fan tightness at the point $x$ of a topological space $(X,\tau)$. This notion is defined as follows: For the point $x\in X$ we define
\[
  \Omega_x = \{A\subseteq X\setminus \{x\}: x \mbox{ is in the closure of }A\}.
\]
Then $(X,\tau)$ is said to have \emph{countable strong fan tightness at} $x$ if the selection principle $\sone(\Omega_x,\Omega_x)$ holds. 

In \cite{COC3} it was shown that for certain ``nice" spaces $(X,\tau)$ it is true that $\sone(\Omega_x,\Omega_x)$ holds if, and only if, ONE has now winning strategy in the game $\gone(\Omega_x,\Omega_x)$. In \cite{COC3} also an ad hoc example of the failure of this equivalence was given. In Theorem 11 of \cite{RCT} the following more extreme example is given, assuming the Continuum Hypothesis (CH):
\begin{theorem}[CH]\label{oneextreme}
There is a $\textsf{T}_3$ space $X$ that has countable strong fan tightness at each $x\in X$, yet ONE has a winning strategy in $\textsf{G}_1(\Omega_x,\Omega_x)$ at each $x\in X$.
\end{theorem}

\section{Lemma 10}

For topological space $(X,\tau)$ let $\mathfrak{D}$ denote the set $\{A\subseteq X:\; A \mbox{ dense in }X\}$. Also, let $\mathfrak{ND}$ denote the set $\{A\subseteq X:\; A \mbox{ is not discrete}\}$.  In \cite{RCT} the proof of Theorem \ref{oneextreme} made use of a lemma regarding the infinite game $\gone(\mathfrak{D},\; \mathfrak{ND})$. This game is a dual version of a game introduced in \cite{BJ} by Berner and Juhasz. More details about this game and construction of the dual game appear in the sources \cite{BJ,  COC6, RCT}.

The claimed proof ot Lemma 10 of \cite{RCT} is flawed. Here is a correct argument:

\begin{lemma}\label{modestlemma}
Let $(X,\tau)$ be a $\textsf{T}_1$-space with no isolated points. Assume that player ONE has a winning strategy in the game $\gone(\mathfrak{D},\mathfrak{ND})$ on $X$. Then at each $x\in X$ ONE has a winning strategy in the game $\gone(\Omega_x,\Omega_x)$.
\end{lemma}
\begin{proof} 
Let $\sigma$ be a winning strategy for ONE in the game $\gone(\mathfrak{D},\mathfrak{ND})$. For each $x\in X$, define a strategy $\sigma_x$ as follows:
\begin{itemize}
\item{$\sigma_x(\emptyset) = \sigma(\emptyset)\setminus\{x\}$, and }
\item{for each finite sequence $(w_1,\;\cdots,\; w_n)$ of elements of $X$, $\sigma_x(w_1,\;\cdots,\; w_n) = \sigma(w_1,\;\cdots,\; w_n)\setminus\{x\}$.}
\end{itemize}
Since $X$ has no isolated points, for each dense set $D\subset X$ and each $x\in X$, the set $D\setminus\{x\}$ is dense in $X$. Thus, $\sigma_x$ is a strategy for player ONE in the game $\gone(\mathfrak{D},\mathfrak{ND})$. 
Each $\sigma_x$ play of the game $\gone(\mathfrak{D},\mathfrak{ND})$ is a $\sigma$-play during which TWO never picked the element $x$. Thus, $\sigma_x$ is also a winning strategy for ONE in $\gone(\mathfrak{D},\mathfrak{ND})$.

Now we note that at each $x\in X$ the strategy $\sigma_x$ is a winning strategy for ONE in the game $\gone(\Omega_x,\Omega_x)$: For consider a $\sigma_x$-play
\[
   O_1,\; w_1,\; O_2,\; w_2,\; \cdots,\; O_n,\; w_n,\; \cdots
\]
For each $n$, $O_n$ is a dense set not containing the point $x$, and thus is an element of $\Omega_x$. The set $\{w_n:0<n<\omega\}$ is a discrete subset of $X$ and does not contain the point $x$. Thus, let $U$ be a neighborhood of $x$ meeting the set of moves by TWO in at most one point, say $w_n$. Since $x\neq w_n$ and $X$ is $\textsf{T}_1$, there is a neighborhood $W$ of $x$ that does not contain $w_n$. But then $U\cap W$ is a neighborhood of $x$ disjoint from $\{w_n:0<n<\omega\}$. It follows that $\{w_n:0<n<\omega\}$ is not an element of $\Omega_x$.
\end{proof}

\section*{Acknowledgements}
I thank Dr. Boaz Tsaban for pointing out the shortcomings in the original argument given for the proof of Lemma 10 of \cite{RCT}.

\end{document}